\documentclass[12pt]{amsart}

\usepackage[all]{xy}

\newtheorem{theorem}{Theorem}[section]
\newtheorem{lemma}[theorem]{Lemma}

\newtheorem{prop}[theorem]{Proposition}

\def\Om{\Omega}       \def\th{\theta}
\def\al{\alpha}       
\def\be{\beta}

\def\vi{\varphi}      \def\io{\iota}
       \def\om{\omega}
        
 \def\mZ{\mathbb Z}	\def\mQ{\mathbb Q}
 \def\kT{\mathcal T}	\def\mF{\mathbb F}

\def\hH{\hat{H}}
\def\tA{\tilde{A}}	\def\tM{\tilde{M}}

\def\niso{\not\simeq}
\def\Tr{\mathop\mathrm{Tr}\nolimits}
\def\rad{\mathop\mathrm{rad}\nolimits}
\def\Hom{\mathop\mathrm{Hom}\nolimits}
\def\ker{\mathop\mathrm{Ker}\nolimits}
\def\im{\mathop\mathrm{Im}\nolimits}
\def\cok{\mathop\mathrm{Coker}\nolimits}

\def\CM{\mathrm{CM}}
\def\lat{\mbox{-}\mathrm{lat}}
\def\ind{\mbox{-}\mathrm{ind}}
\def\xarr{\xrightarrow}
\def\sb{\subset}         \def\sp{\supset}
\def\spe{\supseteq}      \def\sbe{\subseteq}
\def\*{\otimes}   		\def\bop{\bigoplus}
\def\={\setminus}	\def\+{\oplus}

\def\setsuch#1#2{\left\{\,#1\mid #2\,\right\}}

\begin{document}

\def\cm{Cohen--Macaulay}
\def\AR{Auslander--Reiten }
\def\oc{one-to-one correspondence}
\def\iff{if and only if }

\author[Yu. Drozd]{Yuriy A. Drozd}
\title{Cohomologies of the Kleinian 4-group}
\address{National Academy of Sciences of Ukraine \\ Institute of Mathematics\\ Tereschenkivska str. 3\\ 01024 Kyiv\\ Ukraine}
\email{y.a.drozd@gmail.com}
\urladdr{www.imath.kiev.ua/$\sim$drozd}
\author[A. Plakosh]{Andriana I. Plakosh}
\address{National Academy of Sciences of Ukraine \\ Institute of Mathematics\\ Tereschenkivska str. 3\\ 01024 Kyiv\\ Ukraine}
\email{andrianaplakoshmail@gmail.com}



\begin{abstract}
 We calculate explicitly the cohomologies of all $G$-lattices, where $G$ is the Kleinian $4$-group.
\end{abstract}

\maketitle

\section*{Introduction}

 The calculation of cohomologies of a given group is an important and interesting, but usually a cumbersome problem. 
 So only some cases are known where
 such calculations were made for a rather wide class of modules. If the group is finite, a special interest is in cohomologies of
 \emph{lattices}, i.e. $G$-modules which are finitely generated and torsion free as groups. They are of importance, for example, in the
 theory of crystallographic groups and of Chernikov groups. Certainly, if we want to know all cohomologies of lattices, one would like to have
 a classification of $G$-lattices. In the case of finite $p$-groups, such classification is only known for cyclic groups of order $p$ and $p^2$ for a
 prime $p$ \cite[\S\,34B,\,\S\,34C]{cr}, for the cyclic group of order $8$ \cite{yak} and for the Kleinian $4$-group \cite{naz,pl}. In other cases
 such a classification is \emph{wild}, i.e. includes a description of representations of all finite dimensional algebras.
 In the first above mentioned case there are only finitely many indecomposable representations, the cohomologies $\hH^n(G,m)$ are periodic 
 of period $2$, so the answer can be easily obtained. It was used in \cite{gs}, where a complete list of Chernikov $p$-groups with the cyclic top of order 
 $p$ or $p^2$ was obtained. For the Kleinian group the question becomes much more complicated, since, first, there are infinitely many non-isomorphic 
 indecomposable lattices, and, second, the cohomologies are no more periodic. In this paper we use the description of indecomposable lattices from 
 \cite{pl} and some general facts about cohomologies of $p$-groups and give a complete description of cohomologies of lattices over the Kleinian group.
 
 \medskip\noindent
 It is known that for the Kleinian 4-group $G$ every $\mZ_2G$-lattice, where $\mZ_2$ is the ring of $p$-adic integers, coincides with the $2$-adic 
 completion $\hat{M}_2$ of some $\mZ G$-lattice $M$, and $\hat{M}_2\simeq\hat{N}_2$ \iff $M\simeq N$ (cf.~\cite{naz}). 
 As $\hH^n(G,M)=\hH^n(G,\hat{M}_2)$, we only have to consider the $\mZ_2 G$-lattices.

 \section{Preliminaries}
 \label{s1} 
 
 In this section we establish some results on \AR translate and cohomologies of $p$-groups.
 
 \medskip\noindent
Let $R$ be a local complete noetherian ring, equidimensional of Krull dimension 1 and without nilpotent elements, $R\lat$ be the category
of $R$\emph{-lattices}, i.e. maximal \cm\ $R$-modules or, equivalently, torsion free $R$-modules. It is known that $R$ has a  \emph{canonical module} 
$\om_R\in R\lat$ such that the functor $D=\Hom_R(\_\,,\om_R)$ is an exact duality on the category $R\lat$ \cite[Corollary 3.3.8]{bh}.
In particular, it induces a bijection on the set $R\ind$ of isomorphism classes of indecomposable $R$-lattices. 

 Recall that an \emph{\AR sequence} in $R\lat$ is a non-split exact sequence
\begin{equation}\label{e1} 
0\to M'\xarr\be N\xarr\al M\to0,	\tag{AR}
\end{equation}
 where $M$ and $M'$ are indecomposable $R$-lattices and
\begin{itemize}
\item  for every homomorphism $\xi:L\to M$, where $L$ is an indecomposable $R$-lattice and $\xi$ is not an isomorphism, there is
$\xi':L\to N$ such that $\xi=\al\xi'$;
\item for every homomorphism $\eta:M' \to L$, where $L$ is an indecomposable $R$-lattice and $\eta$ is not an isomorphism, there is
$\eta':N\to L$ such that $\eta=\eta'\be$.\!%
\footnote{\,It is known (see, for instance, \cite{ars}) that each of these two conditions implies the other.}
\end{itemize}
In this case $M$ and $M'$ define one another up to an isomorphism. They are denoted: $M'=\tau_RM$ and $M=\tau_R^{-1}M'$.
$\tau_R$ is called the \emph{\AR translate} for $R$-lattices.
It is known \cite{ar} that for every indecomposable $R$-lattice $M\niso R$ (for every indecomposable $R$-lattice $M'\niso \om_R$)
there is an \AR sequence \eqref{e1}, so $\tau_RM$ (respectively, $\tau^{-1}_RM'$) is defined. If $R$ is Gorenstein, i.e. $\om_R\simeq R$,
$\tau_R$ induces a bijection on the set $R\ind{\=}\{R\}$. 

\medskip\noindent
 Let in the \AR sequence~\eqref{e1} $N=\bop_{i=1}^mN_i$, where $N_i$ are indecomposable, $\al_i:N_i\to M$ and $\be_i:M'\to N_i$ be the components
 of $\al$ and $\be$ with respect to this decomposition. The \emph{\AR quiver} of the category $R\lat$ is a quiver whose vertices are the isomorphism 
 classes of indecomposable $R$-lattices and arrows are given by the following rules:
 \begin{itemize}
 \item If $M\niso R$, the arrows ending in $M$ are just $\al_i:N_i\to M$.
 \item If $M'\niso \om_R$, the arrows starting in $M$ are just $\be_i:M'\to N_i$.
 \item If $\rad R=\bop_{i=1}^kQ_i$, where $Q_i$ are indecomposable, the arrows ending in $R$ are just the embeddings $\io_i:Q_i\to R$.
 \item The arrows starting in $\om_R$ are just $D\io_i:DQ_i\to\om_R$.
 \end{itemize}

\noindent
We denote by $\Om M$ the \emph{syzygy} of $M$, i.e. the kernel of an epimorphism $\vi:P\to M$, where $P$ is projective and
$\ker\vi\sbe\rad P$. Again, if $R$ is Gorenstein, $\Om$ induces a bijection on the set $R\ind{\=}\{R\}$. So in this case $\Om^{-1}M$ is well 
defined for any indecomposable lattice $M\niso R$.

\begin{prop}\label{p1} 
If $R$ is Gorenstein, $\tau_RM\simeq \Om M$ for every indecomposable $R$-lattice $M\niso R$.
\end{prop}
\begin{proof}
  For an indecomposable $R$-lattice $M\niso R$, consider a minimal projective presentation $P_1\xarr\psi P_0\xarr\vi M\to0$, i.e. an
  exact sequence, where $P_0,P_1$ are projective and both $\ker\vi\sbe\rad P_0$ and $\ker\psi\sbe\rad P_1$. Then $\ker\vi=\Om M$. 
  Let $N=\ker\psi$. Set $\Tr M=\cok D\psi=DN$. Then there is an \AR sequence \eqref{e1}, where $M'=D\Om\Tr M$ \cite{ar}. Since the
  exact sequence $DP_0\xarr{D\psi}DP_1\to DN\to 0$ is a minimal projective resolution of $DN$, we have that $\Om\Tr M=\im D\psi$ and
  $D\Om\Tr M=D(\im D\psi)=\im\psi=\ker\vi=\Om M$.
\end{proof}

\noindent
 Recall that if $R$ is Gorenstein and not regular there is a unique ring $A\sp R$ such that $A/R\simeq R/\rad R$, $A\in R\lat$ (hence $A\lat$ is a full
 subcategory of $R\lat$) and every indecomposable $R$-lattice $M\niso R$ is an $A$-lattice. $A$ is called the \emph{minimal overring} of $R$ 
 (see \cite{dr}). By duality, $DA\simeq\rad R$.

\begin{prop}\label{p2} 
 Let the ring $R$ be Gorenstein and non-regular, $A$ be the minimal overring of $R$.
 \begin{enumerate}
  \item If $0\to M'\xarr\be N\xarr\al M\to0$ is an \AR sequence in $A\lat$,  it is also an \AR sequence in $R\lat$.
  \item Let $M$ be an indecomposable $A$-lattice. If $M\niso A$, then $\tau_RM=\tau_AM$, and if $M\niso\om_A $, then
  $\tau^{-1}_RM=\tau^{-1}_AM$.
  \end{enumerate} 
\end{prop}
\begin{proof}
 (1) Let $L$ be an indecomposable $R$-lattice, $\xi:L\to M$ be not an isomorphism. If $L\niso R$, it is an $A$-lattice, hence there is a homomorphism 
 $\xi':L\to N$ such that $\xi=\al\xi'$. If $L=R$, such a homomorphism exists since $R$ is projective. Let now $\eta:M'\to L$ be not an
 isomorphism. Again, if $L\niso R$, there is $\eta':N\to L$ such that $\eta=\eta'\be$. If $L=R$, $\im\be\in\rad R$ and $\rad R\in A\lat$,
 which implies again the existence of $\eta'$.
 
 (2) is an immediate consequence of (1).
\end{proof}

\noindent
From now on $R=\mZ_pG$, where $G$ is a finite commutative $p$-group. It is local, Gorenstein, non-regular and $R/\rad R\simeq\mF_p$
(with the trivial action). Let $A$ be the minimal overring of $R$. 
 
 \begin{prop}\label{p5} 
  $\tau_R A\simeq \om_A$.
 \end{prop}
 \begin{proof}
  Otherwise $\tau_RA\simeq M$ for some $M\niso\om_A$. Hence $A=\tau_R^{-1}M=\tau_A^{-1}M$, which is impossible, since
  $A$ is a projective $A$-module.
 \end{proof}
 
\begin{prop}\label{p3} 
 Let $G=\prod_{i=1}^sG_i$, where $G_i$ are cyclic groups.
 \[
  \hH^n(G,A)\simeq \hH^n(G,\mF_p)\simeq
  \begin{cases} \displaystyle
   \binom{n+s-1}{s-1}\mF_p &\text{\rm if } n\ge0,\\
   \displaystyle
   \binom{s-n-2}{s-1}\mF_p &\text{\rm if } n<0.
  \end{cases}
 \]
\end{prop}
\begin{proof}
 The first isomorphism follows from the exact sequence $0\to R\to A\to \mF_p\to 0$, since $\hH^n(G,R)=0$. The cohomologies of $\mF_p$
 can be easily calculated using the free resolution of $\mZ$ described in \cite{dp}. Namely, the latter has a free $R$-module of rank  
 $\binom{n+s-1}{s-1}$ as the $n$-th component and the image of the differential $d_n$ is in the radical, which implies the result.
\end{proof}

\begin{prop}\label{p4} Let $M\in\CM(R),\,M\niso R$. Then
\[
 \hH^n(G,M)\simeq\hH^{n+1}(G,\tau_RM)\simeq\hH^{n-1}(G,\tau^{-1}_RM)
\]
\end{prop}
\begin{proof}
 It follows immediately from Proposition~\ref{p1}.
\end{proof}

\noindent
 Now $G$ is the Kleinian $4$-group: $G=\langle a,b\,|\,a^2=b^2=1,\,ab=ba\rangle$. Then $\mZ_2G$ has $4$ irreducible lattices $L_{uv}$, 
 where $u,v\in\{+,-\}$. Namely, $L_{uv}=\mZ_2$ as $\mZ_2$-module, $a$ acts as $u1$ and $b$ acts as $v1$. 
  
\begin{prop}\label{p6} 
\begin{align*}
  \hH^n(G,L_{++})&=
  \begin{cases}
   (|n|/2+1)\mF_2 &\text{\rm if $n\ne0$ is even},\\
  [|n|/2] \mF_2 &\text{\rm if $n$ is odd},\\
   \mZ/4\mZ &\text{\rm if } n=0,
  \end{cases} \\
  \intertext{and if $(u,v)\ne(+,+)$, then}
 \hH^n(G,L_{uv})&= ([(|n|+1)/2])\mF_2
\end{align*}
\end{prop}
\begin{proof}
 It is a partial case of \cite[Theorem 4.3 \& Corollary 4.2]{dp}.\!%
 \footnote{\,Note that there is an obvious misprint in \cite[Theorem~4.3]{dp}: in the formula (4.4) there must be $|n|-2$ instead of $|n|-1$.}
\end{proof}

 \section{Cohomologies}
 \label{s2} 
 
 In this section $R=\mZ_2G$, where $G$ is the Kleinian $4$-group, $A$ is the minimal overring of $R$. Recall the structure of the \AR quiver 
 of $A\lat$. It follows from \cite{rog} and \cite{pl} that this quiver consists of the \emph{preprojective-preinjective} component and \emph{tubes}. 
 The preprojective-preinjective component has the form
 {\footnotesize
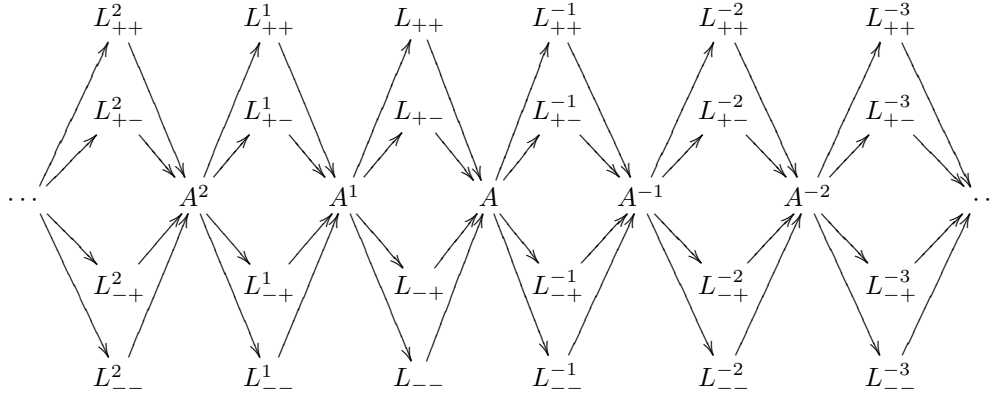
 \[
  \xymatrix@C=.65em@R=1.7em{ & L_{++}^2 \ar[ddr] && L_{++}^1 \ar[ddr] && L_{++} \ar[ddr] && L_{++}^{-1} \ar[ddr] && L_{++}^{-2} \ar[ddr]
  && L_{++}^{-3} \ar[ddr] \\
  & L_{+-}^2 \ar[dr] && L_{+-}^1 \ar[dr] && L_{+-} \ar[dr] && L_{+-}^{-1} \ar[dr] && L_{+-}^{-2} \ar[dr]
  && L_{+-}^{-3} \ar[dr] \\
 {\cdots\ \ {}} \ar[uur] \ar[ur] \ar[ddr] \ar[dr] && A^2 \ar[uur] \ar[ur] \ar[ddr] \ar[dr] && A^1 \ar[uur] \ar[ur] \ar[ddr] \ar[dr]  && A \ar[uur] \ar[ur] \ar[ddr] \ar[dr] 
  && A^{-1} \ar[uur] \ar[ur] \ar[ddr] \ar[dr] && A^{-2} \ar[uur] \ar[ur] \ar[ddr] \ar[dr] && {{}\ \ \cdots}\\
  & L_{-+}^2 \ar[ur] && L_{-+}^1 \ar[ur]  && L_{-+} \ar[ur] && L_{-+}^{-1} \ar[ur] && L_{-+}^{-2} \ar[ur]
  && L_{-+}^{-3} \ar[ur] \\
  & L_{--}^2 \ar[uur] && L_{--}^1 \ar[uur]  && L_{--} \ar[uur] && L_{--}^{-1} \ar[uur] && L_{--}^{-2} \ar[uur]
  && L_{--}^{-3} \ar[uur]
   }
 \]
 }
 Here $M^k$ denotes $\tau_R^kM$. In particular, $A^1=\om_A$. Propositions~\ref{p4} and~\ref{p6} imply the following 
 values of cohomologies of these lattices.
 \begin{theorem}\label{pj} 
 \begin{align*}
   \hH^n(G,A^k)&=
  \begin{cases}
   (n-k+1)\mF_2 &\text{\rm if } n\ge k,\\
   (k-n)\mF_2 &\text{\rm if }n<k;
  \end{cases}\\
   \hH^n(G,L_{++}^k)&=
  \begin{cases}
  (|n-k|/2+1) \mF_2 &\text{\rm if $n-k\ne0$ is even},\\
  [|n-k|/2]\mF_2 &\text{\rm if $n-k$ is odd},\\
   \mZ/4\mZ &\text{\rm if $n=k$};
  \end{cases}\\
  \intertext{and if $(u,v)\ne(+,+)$, then}
   \hH^n(G,L_{uv}^k)&=[(|n-k|+1)/2]\mF_2. && \qed
 \end{align*} 
 \end{theorem}

\medskip\noindent
  To calculate the cohomologies of the lattices that belong to tubes, we need the following considerations.
 Let $\tA$ be the integral closure of $A$ in $A\*_{\mZ_p}\mQ_p$. For any $A$-lattice $M$ the lattice $\tM=\tA M$ is a direct sum
 $\bop_{u,v\in\{+,-\}}M_{uv}$, where $M_{uv}\simeq (r_{uv}M)L_{uv}$ for some integers $r_{uv}M$. 
 Note that every $\mZ_2$-submodule of $M_{uv}$ is actually an $\tA$-submodule.
 
 \begin{lemma}\label{lem} 
  If $M\niso L_{++}$ is an indecomposable $R$-lattice, then $\hH^0(G,M)=(r_{++}M)\mF_2$\!\!.
 \end{lemma}
 \begin{proof}
  Recall that $\hH^0(G,M)=M^G/tr M$, where $M^G$ is the set of invariants: $M^G=\setsuch{m\in M}{gm=m \text{ for all } g\in G}$, 
  and $tr=\sum_{g\in G}g$.
  Note that $\rad A=2\tA$, hence $2\tM=\rad M\sb M\sbe\tM$. Therefore, $M^G\spe 2M_{++}$. Suppose that $m\in M^G\=2M_{++}$.
  Then $\pi(m)=m'\in M_{++}\=2M_{++}$, where $\pi:M\to M_{++}$ is the projection. Hence $M_{++}=\mZ_2m'\+N$ for some submodule $N\sb M_{++}$. 
  Let $\th:M\to\mZ_2m'$ be the composition of $\pi$ and the projection $M_{++}\to\mZ_2m'$. There is a homomorphism $\eta:\mZ_2m'\to M$
  which maps $m'$ to $m$. Then $\th\eta$ is identity on $\mZ_2m'$, so $M\simeq\mZ_2m\+\ker\th$. As $M$ is indecomposable,
  $M=\mZ_2m\simeq L_{++}$, which is impossible. 
  
  \medskip\noindent
  Therefore, $M^G=2M_{++}$. On the other hand, as $\im\pi=\tA\im\pi$, it coincides with the image of the projection $\tM\to M_{++}$, 
  which is $M_{++}$. Hence $\pi$ is a surjection and its restriction onto $tr M$ is also a surjection $trM\to trM_{++}=4M_{++}$. 
  As $tr M\sbe M^G=2M_{++}$, it implies that $trM=4M_{++}$ and $\hH^0(G,M)=2M_{++}/4M_{++}\simeq(r_{++}M)\mF_2$.
 \end{proof}
 
  \noindent
  Now recall the structure of tubes \cite{pl}.
 \emph{Homogeneous tubes} are parametrized by irreducible unital polynomials $f\in\mF_2[x]$, $f\notin\{x,x-1\}$. The tube $\kT^f$ is of the form
  \[
  \xymatrix{  T_1^f \ar@/^/[r] & T_2^f \ar@/^/[r] \ar@/^/[l] & T_3^f \ar@/^/[r] \ar@/^/[l] & T_4^f \ar@/^/[r] \ar@/^/[l] & \ar@/^/[l] \cdots
  }
  \]
 where $r_{++}(T^f_k)=kd$ for $d=\deg f$ and $\tau_RT_k^f=T_k^f$.
 
 \medskip\noindent
 There are also $3$ \emph{special tubes} $\kT^j\ (j\in\{2,3,4\})$. They are of the form
 \[
  \xymatrix{ T^{j1}_1 \ar[r]  & T^{j1}_2 \ar[r] \ar[dl] & T^{j1}_3 \ar[r] \ar[dl]& T^{j1}_4 \ar[r] \ar[dl] \ar[r] \ar[dl] & \cdots \ar[dl]\\
   T^{j2}_1 \ar[r]  & T^{j2}_2 \ar[r] \ar[ul] & T^{j2}_3 \ar[r] \ar[ul]& T^{j2}_4 \ar[r] \ar[ul] \ar[r] \ar[ul] & \cdots \ar[ul]}
 \]
 where $r_{++}T^{j1}_{2k}=r_{++}T^{j2}_{2k}=k$, $r_{++}T^{j1}_{2k-1}=k,\,r_{++}T^{j2}_{2k-1}=k-1$ and 
  $\tau_RT^{j1}_k=T^{j2}_k,\,\tau_RT^{j2}_k=T^{j1}_k$. 
 
 \medskip\noindent
 Using Lemma~\ref{lem} and Proposition~\ref{p4}, we obtain the following result that accomplishes the calculation of cohomologies of
 lattices over the Kleinian group.
 
 \begin{theorem}\label{tube}  
 \begin{align*}
   \hH^n(G,T_k^f)&=kd\,\mF_2, \text{ \rm where } d=\deg f,\\
   \intertext{and for every $j\in\{2,3,4\}$}
   \hH^n(G,T^{ji}_{2k})&=k\,\mF_2 \text{ \rm for both } i=1,2,\\
   \hH^n(G,T^{j1}_{2k-1})&=
   \begin{cases}
    k\,\mF_2 &\text{\rm if $|n|$ is even},\\
    (k-1)\mF_2 &\text{\rm if $|n|$ is odd},
   \end{cases}\\
   \hH^n(G,T^{j2}_{2k-1})&=
   \begin{cases}
    k\,\mF_2 &\text{\rm if $|n|$ is odd},\\
    (k-1)\mF_2 &\text{\rm if $|n|$ is even}.
   \end{cases}\\
 \end{align*}
 \end{theorem}

\end{document}